\documentclass[reqno,final]{amsart} 

\usepackage[latin1]{inputenc}

\usepackage[notref,notcite,color]{showkeys}

\usepackage{enumitem}
\setenumerate{leftmargin=*}

\usepackage{amsmath,amssymb,bbm}
\usepackage{color,epsfig,psfrag,url}

\usepackage[colorlinks,bookmarks]{hyperref}

\let\version\today

\let\subset\subseteq
\let\eps\varepsilon

\def\dcup{\dot\cup}

\def\colhom{\stackrel{\sigma}{\to}}

\def\itmit#1{\rm ({\it #1\,})}
\def\rom{\itmit{\roman{*}}}
\def\abc{\itmit{\alph{*}}}

\def\cF{\mathcal{F}}
\def\cG{\mathcal{G}}
\def\cS{\mathcal{S}}

\def\cC{\mathcal{C}}
\def\cP{\mathcal{P}}  
\def\cR{\mathcal{R}}  

\def\Prob{\mathbb{P}}
\def\eqed{\eqno\qed}

\newtheorem{theorem}                   {Theorem}
\newtheorem{lemma}           [theorem] {Lemma}   
   
\newtheorem{proposition}     [theorem] {Proposition}

\newtheorem{definition}      [theorem] {Definition} 
 
\theoremstyle{remark}

\newcommand{\makeversion}{
  \let\oldfootnote\thefootnote %
  \renewcommand{\thefootnote}{} %
  \footnote{\version} %
  \let\thefootnote\oldfootnote %
}

\newcommand{\By}[2]{\overset{\mbox{\tiny{#1}}}{#2}}    
\newcommand{\ByRef}[2]{   \By{\eqref{#1}}{#2} }        

\newcommand{\leBy}[1]{    \By{#1}{\le} }

\newcommand{\eqByRef}[1]{ \ByRef{#1}{=} }

\newcommand{\leByRef}[1]{ \ByRef{#1}{\le} }
\newcommand{\geByRef}[1]{ \ByRef{#1}{\ge} }

\newcommand{\subref}[2][L]{\text{\tiny #1\ref{#2}}}

\newcommand{\field}[1]{\mathbb{#1}}
\newcommand{\N}{\field{N}}

\newcommand{\R}{\field{R}}

\newcommand{\Forb}{\mathcal{F}orb}
\newcommand{\PP}{\cP}
\DeclareMathOperator{\ex}{ex}
\DeclareMathOperator{\UB}{UB}

\newcommand{\BAR}[1]{{\smash{\overset{\scalebox{2}[1.2]{\!\raisebox{-1.3pt}{\_}}}{\smash{#1}\vphantom{X}}}}}
\newcommand{\antiP}{\BAR{P}_3}
\newcommand{\bleft}[1]{\left#1 \vphantom{2^2}}

 \title[Perfect graphs of fixed density]{Perfect graphs of fixed density: \\ counting
   and homogenous sets}

  \author{Julia B\"ottcher}

  \address{
    Instituto de Matem\'atica e Estat\'{\i}stica, Universidade de
    S\~ao Paulo, Rua do Mat\~ao 1010, 05508--090~S\~ao Paulo, Brazil}
  \email{julia@ime.usp.br}

  \author{Anusch Taraz } 
  \address{Zentrum Mathematik, Technische Universit\"at M\"unchen, 
    Boltzmannstra\ss{}e~3, D-85747 Garching bei M\"unchen, Germany} 
  \email{taraz@ma.tum.de}

  \author{Andreas W\"urfl }
  \address{Zentrum Mathematik, Technische Universit\"at M\"unchen, 
    Boltzmannstra\ss{}e~3, D-85747 Garching bei M\"unchen, Germany} 
  \email{wuerfl@ma.tum.de}

  \thanks{ 
    The authors were partially supported by DFG grant TA 309/2-2.
    The first author was partially supported by FAPESP
    (Proc.~2009/17831-7), and is grateful to NUMEC/USP, N\'ucleo de Modelagem
    Estoc\'astica e Complexidade of the University of S\~ao Paulo, for
    supporting this research.
  }

\begin{document}


\begin{abstract}
  For $c \in [0,1]$ let $\cP_n(c)$ denote the set of $n$-vertex perfect
  graphs with density $c$ and $\cC_n(c)$ the set of $n$-vertex graphs
  without induced $C_5$ and with density $c$.  We show that $\log_2
  |\cP_n(c)|/\binom{n}{2}=\log_2 |\cC_n(c)|/\binom{n}{2}=h(c) + o(1)$ with
  $h(c)=\tfrac12$ if $\tfrac14\le c\le \tfrac34$ and $h(c)=\tfrac12
  H(|2c-1|)$ otherwise, where $H$ is the binary entropy function.

  Further, we use this result to deduce that almost all graphs in
  $\cC_n(c)$ have homogenous sets of linear size. This answers a question
  raised by Loebl, Reed, Scott, Thomason, and Thomass\'e\ [Almost all
  $H$-free graphs have the Erd\H{o}s-Hajnal property] in the case of
  forbidden induced $C_5$.
\end{abstract}

\maketitle
\let\languagename\relax   


\section{Introduction and Results}
\label{sec:intro}

In this paper we investigate classes of graphs that are defined by forbidding certain
substructures. Let $\cP$ be such a class. We focus on two related goals: to approximate the cardinality 
of $\cP$ and to determine the structure of a typical graph in $\cP$. 
In particular, we add the additional constraint that all graphs in $\cP$ must have the same density 
$c$ and would like to know how the answer to these questions depends on the parameter $c$.  

The quantity $|\cP_n|$ where $\cP_n:=\{G\in\cP: V(G)=[n]\}$ is also called
the \emph{speed} of~$\cP$.  Often exact formulas or good estimates for
$|\cP_n|$ are out of reach. In these cases, however, one might still ask
for the asymptotic behaviour of the speed of~$\cP$. One prominent result in
this direction was obtained by Erd\H{o}s, Frankl and
R\"odl~\cite{ErdFraRod} who considered properties $\Forb(F)$ defined by a
single forbidden (weak) subgraph $F$. They proved that for each graph~$F$
with $\chi(F)\ge 3$ the class $\Forb_n(F)$ of $n$-vertex graphs that do not
contain~$F$ as a subgraph satisfies $|\Forb_n(F)| = 2^{\ex(F,n)+o(n^2)}$
where $\ex(F,n):=(\chi(F)-2)\binom
n2/(\chi(F)-1)$. 
In other words, if $\chi(F)\ge 3$ then the speed of $\Forb(F)$
asymptotically only depends on the chromatic number of~$F$.

In this paper we are interested in features of the picture at a more fine
grained scale.
More precisely, we fix a density $0<c<1$ and are interested in the number
$|\cP_n(c)|$ of graphs on $n$ vertices with property~$\cP$ and density $c$.
Let $\Forb_n(F,c) = \Forb_n(F)\cap \cG_n(c)$ 
where $\cG_n(c)$ is the set of all graphs on vertex set $[n]$ with $c\binom n2$
edges. Straightforward modifications of the proof of the theorem of Erd\H{o}s,
Frankl and R\"odl~\cite{ErdFraRod} yield the following bounds for
$|\Forb_n(F,c)|$ (we will sketch this argument in Section~\ref{sec:subgraph}):
%
%
  Let $F$ be a graph with $\chi(F)=r$. For all $c\in (0,\tfrac{r-2}{r-1})$ we
  have 
  \begin{equation}\label{eq:subgraph}
  	\lim_{n \to \infty} \frac{\log |\Forb_n(F,c)|}{\binom{n}{2}} =
  	\tfrac{r-2}{r-1}\, H\Big(\tfrac{r-1}{r-2}c\Big) \, ,
  \end{equation}
  where $H(x)$ is the \emph{binary entropy} function, that is,
  for~$x\in(0,1)$ we set $H(x):=-x\log x-(1-x)\log(1-x)$. Here we denote by $\log$ the logarithm to base $2$.
%
Notice that $\lim_{n \to \infty} \log|\Forb_n(F,c)|/\binom{n}{2}=0$ for
$c\ge\tfrac{r-2}{r-1}$ by the theorem of Erd\H{o}s and Stone~\cite{ErdSto}.

The analogous problem for a graph class $\Forb^*(F)$, characterised by a
forbidden \emph{induced} subgraph~$F$, is more challenging and was first
considered by Pr\"omel and Steger~\cite{ProSte_indIII}. They specified a graph
parameter, the so-called colouring number $\chi^*(F)$ of~$F$, that serves as a
suitable replacement of the chromatic number in the theorem of Erd\H{o}s, Frankl
and R\"odl. More precisely, they showed that
$|\Forb^*_n(F)|=2^{\ex^*(F,n)+o(n^2)}$ with
$\ex^*(F,n):=\big(\chi^*(F)-2\big)\binom n2/\big(\chi^*(F)-1\big)$ where
$\chi^*(F)$ is defined as follows. A \emph{generalised $r$-colouring} of $F$ with
$r'\in[0,r]$ cliques is a partition of $V(F)$ into $r'$ cliques and $r-r'$
independent sets. The \emph{colouring number} $\chi^*(F)$ is the largest integer
$r+1$ such that there is an $r'\in[r]$ for which $F$ has no generalised
$r$-colouring with $r'$ cliques. For example, we have $\chi^*(C_5)=3$ and
$\chi^*(C_7)=4$.


This naturally extends to hereditary graph properties, i.e., classes of
graphs~$\cP$ which are closed under isomorphism and taking induced subgraphs (and
may therefore be characterised by possibly infinitely many forbidden induced
subgraphs). Let~$\cF(r,r')$ denote the family of all graphs that admit a
generalised $r$-colouring with $r'$ cliques. Then the colouring number of~$\cP$
is
\[ \chi^*(\PP) := \max\{r+1\colon \cF(r,r')\subset\cP\text{ for some } r'
\in[0,r] \}\;, \] 
and we set $\ex^*(\cP,n):=\big(\chi^*(\cP)-2\big)
\binom{n}{2}/\big(\chi^*(\cP)-1\big)$. Observe that this definition implies
$\chi^*\big(\Forb^*(F)\big)=\chi^*(F)$. And indeed Alekseev~\cite{Alek92}, and
Bollob\'as and Thomason~\cite{BT97} generalised the result of Pr\"omel and Steger
to arbitrary hereditary graph properties~$\cP$ and showed that
$|\PP_n|=2^{\ex^*(\cP,n)+ o(n^2)}$.

More precise estimates for the speed were given for monotone properties~$\cP$
(properties that are closed under isomorphisms and taking subgraphs) by Balogh,
Bollob\'as, and Simonovits~\cite{BBS04} who showed that
$2^{\ex^*(\cP,n)}\le|\cP_n|\le 2^{\ex^*(\cP,n)+cn\log n}$ for some constant~$c$,
and for hereditary properties~$\cP$ by Alon, Balogh, Bollob\'as, and Morris
\cite{ABBM09} who proved $2^{\ex^*(\cP,n)}\le|\cP_n|\le
2^{\ex^*(\cP,n)+n^{2-\eps}}$ for some $\eps=\eps(\cP)>0$ and~$n$ sufficiently
large. Pr\"omel and Steger~\cite{ProSte_C4,AlmostAllBergeGraphs} gave even more
precise results for the speed of $\Forb^*_n(C_4)$ and $\Forb^*_n(C_5)$ which they
determined up to a factor of $2^{O(n)}$. In fact, they showed
in~\cite{AlmostAllBergeGraphs} that almost all graphs in $\Forb^*_n(C_5)$ are
generalised split graphs, that is, graphs of a rather simple structure which are
defined as follows. We say that a graph $G=(V,E)$ admits a \emph{generalised
clique partition} if there is a partition $V=V_1\dcup\dots\dcup V_k$ of its
vertex set such that $G[V_i]$ is a clique and for $i>j>1$ we have
$e(V_i,V_j)=e(V_j,V_i)=0$. A graph~$G$ is a \emph{generalised split graph} if
$G$ or its complement admit a generalised clique partition.

It is illustrative to compare this result to the celebrated strong perfect
graph theorem~\cite{StrongPerfectGraphTheorem}. 
A graph~$G$ is \emph{perfect} if $\chi(G')$ equals the clique number
$\omega(G')$ for all induced subgraphs~$G'$ of~$G$. The strong perfect graph
theorem asserts that \emph{all} graphs without induced copies of odd cycles
$C_{2i+1}$, $i>1$ and without induced copies of their complements $\overline{
C_{2i+1}}$ are perfect. 
Using this characterisation, it is easy to see that generalised split
graphs are perfect. Consequently the result of Pr\"omel and Steger
implies that already \emph{almost all} graphs without induced $C_5$ are perfect
(observe that $C_5$ is self-complementary).

In this paper, we consider induced $C_5$-free graphs of density $c$
and provide bounds for their number. 
In the spirit of the result by Pr\"omel and Steger we also relate this quantity
to the number of $n$-vertex perfect graphs and generalised split graphs with
density~$c$.

\begin{definition}
  We define the following graph classes:
  \begin{align*}
    \cC(n,c) &:=\Forb^*_n(C_5,c):= \Forb^*_n(C_5)\cap \cG_n(c) \,, \\  
    \cP(n,c) &:=\big\{G\in\cG_n(c)\colon G \text{ is perfect}\big\}\,, \\
    \cS(n,c) &:=\big\{G\in\cG_n(c)\colon G \text{ is a generalised split
       graph}\big\}\,.
  \end{align*}
\end{definition}

Observe that for all~$n$ and~$c\in[0,1]$ we have
$\cS(n,c)\subset\cP(n,c)\subset\cC(n,c)$. Our first main result now bounds the
multiplicative error term between $|\cS(n,c)|$ and $|\cC(n,c)|$. In order to
state this we define the following function.
Let
\begin{align}
\label{eq:h}
  h(c) := 
  \begin{cases} 
    H(2c)/2 & \text{ if } 0<c<\tfrac14\,, \\
    1/2 & \text{ if } \tfrac14\le c\le \tfrac34\,, \\
    H(2c-1)/2 & \text{ otherwise\,.}
  \end{cases}
\end{align}
Note that the classes of all generalised split graphs, all perfect graphs,
and all graphs without induced $C_5$ are closed under taking complements.
Hence, e.g., $|\cC(n,c)|=|\cC(n,1-c)|$ for all $c\in (0,1)$ and $h$ 
is in fact symmetric in $(0,1)$. 
Further note that $H(|2c-1|)/2=h(c)$ for $c<1/4$ or $c>3/4$.

\begin{theorem}
\label{thm:main}
  For all~$c \in (0,1)$ we have
  \[ 
  \lim_{n \to \infty} \frac{\log_2 |\cC(n,c)|}{\binom{n}{2}} =
  \lim_{n \to \infty} \frac{\log_2 |\cP(n,c)|}{\binom{n}{2}} =
  \lim_{n \to \infty} \frac{\log_2 |\cS(n,c)|}{\binom{n}{2}} = h(c)
  \,.
  \]
\end{theorem}

The proof of this theorem uses Szemer\'edi's regularity lemma and is given in
Section~\ref{sec:count}.

We remark that Bollob\'as and Thomason~\cite{BolTho00} studied related
questions of a more general type (see also the references
in~\cite{BolTho00} for earlier results in this direction).  They were
interested in the probability $\Prob_\cP := \Prob[\cG(n,p) \in \cP]$
of an arbitrary hereditary property~$\cP$ in the probability space
$\cG(n,p)$ and showed that for any~$\cP$ there are very simple
properties~$\cP^*$ which closely approximate~$\cP$ in the probability space
$\cG(n,p)$. 
In this context, our Theorem~\ref{thm:main} estimates the probability of
$\cP=\Forb^*_n(C_5)$ in the probability space $\cG(n,m)$ with
$m=c\binom{n}{2}$ and states that $\cP=\Forb^*_n(C_5)$ is
approximated by the property~$\cP^*$ of being a generalised split graph in
$\cG(n,m)$.  
The actual value of the probability~$\Prob_\cP$ was estimated by
Marchant and Thomason in~\cite{MarTho_note} for $\cP=\Forb^*_n(K_{3,3})$ and
$\cP=\Forb^*_n(C^*_6)$, where $C^*_6$ is the six-cycle with a single
diagonal added. They also pointed out the relevance of extremal
properties of coloured multigraphs to problems of this kind (see
also~\cite{MarTho_extcol}). It seems that $\Prob[\cG(n,p) \in
\cP]$ and $\Prob[\cG(n,m=p\binom n2) \in \cP]$ are closely related, at least
for certain properties $\cP$.

%

\medskip
Let us now move from the question of approximating cardinalities to determining
the structure of a typical element in $\Forb_n^*(C_5)$. A well-known conjecture
by Erd\H{o}s and Hajnal \cite{RamseyTypeTheorems} states
that any family of graphs that does not contain a certain fixed graph $H$ as an
induced subgraph must contain a homogenous set whose size is polynomial in the
number of vertices. 

The conjecture is known to be true for certain graphs $H$,
but open, among others, for $H=C_5$ (see \cite{ReflectionsonErdoesHajnal}).
However, Loebl, Reed, Scott, Thomason, and Thomass\'e~\cite{LoeReeScoThoTho} 
recently showed that for any graph $H$ \emph{almost all graphs} in $\Forb_n(H)$ have a
polynomially sized homogenous set. Moreover, they ask for which graphs $H$ it is true that
almost all graphs in $\Forb_n(H)$ do indeed have a \emph{linearly} sized
homogenous set.

It may seem at first sight that 
our estimates derived in Theorem~\ref{thm:main}, carrying an $o(n^2)$ term
in the exponent, are too rough to tell us something about the structure of
almost all graphs in $\Forb_n^*(C_5)$ or $\Forb_n^*(C_5,c)$.
However, we can combine them 
with the ideas of~\cite{LoeReeScoThoTho} 
to answer the question of Loebl, Reed, Scott, Thomason, and Thomass\'e
in the affirmative for the case
$H=C_5$. In fact, we can prove this assertion even in the case where we, again, 
restrict the class to graphs with a given density.



\begin{theorem}\label{thm:C5LinHom}
For $\eta>0$ denote by $\Forb^*_{n,\eta}(F,c)$ the set of graphs
$G\in\Forb^*_n(F,c)$ with $\hom(G)< \eta n$. Then for every $0<c<1$ there
exists $\eta>0$ such that
\[
	\frac{|\Forb^*_{n,\eta}(C_5,c)|}{|\Forb^*_n(C_5,c)|} \to 0 \quad (n\to
	\infty). 
\]
\end{theorem}

We provide the proof of this theorem in Section~\ref{sec:linear}.

Similar statements as in Theorem~\ref{thm:main} and~\ref{thm:C5LinHom}, for
forbidden graphs $F$ other than $C_5$, seem to require more work.

\section{The proof of Theorem~\ref{thm:main}}
\label{sec:count}

In this section we prove Theorem~\ref{thm:main}. In Section~\ref{sec:lower}
we start with the lower bound, by estimating the number of generalised
split graphs with a given density. 
For the upper bound we need some preparations: We shall apply
Szemer\'edi's regularity lemma, which is introduced in
Section~\ref{sec:reg}. In Section~\ref{sec:subgraph} we illustrate how
this lemma can be used for counting graphs without a fixed (not necessarily
induced) subgraph. In Section~\ref{sec:emb} we explain how to modify these
ideas in order to deal with forbidden induced subgraphs. In
Section~\ref{sec:upper}, finally, we prove the upper bound of
Theorem~\ref{thm:main}.

\subsection{The lower bound of Theorem~\ref{thm:main}}
\label{sec:lower}

In this section we estimate the number of generalised split graphs with
density~$c$ and prove the following lemma which constitutes the lower bound of
Theorem~\ref{thm:main}.

\begin{lemma}
\label{lem:lower}
  For all $c,\gamma\in(0,1)$ there is~$n_0$ such that for all
  $n\ge n_0$ we have 
  \[ |\cS(n,c)|\ge 2^{h(c)\binom{n}{2}-\gamma \binom n2}\,. \]
\end{lemma} 

We will use the following bound for binomial coefficients (see, e.g.,
\cite{Juk01}).
For every $\gamma>0$ there exists $n_0$ such that for every integer $m\ge n_0$
and for every real $c\in (0,1)$ we have
\begin{align}\label{eq:BinCoe}
	2^{m H(c) - \gamma m} \le \binom m{cm} \le 2^{m H(c)}\; .
\end{align}
We call the term $-\gamma m$ in the first exponent the \emph{error term of
Equation~\eqref{eq:BinCoe}}.

\begin{proof}[Proof of Lemma~\ref{lem:lower}]
We prove this lower bound by constructing an adequate number of generalised
split graphs. 
Choose $n_0$ sufficiently large such that~\eqref{eq:BinCoe} holds
for $m = \tfrac 12 \binom n2$ and error term $\gamma m$. 
Observe that it suffices to prove the lemma for $c\le\frac12$, since
the complement of a split graph with density~$c$ is a split graph with
density $(1-c)$. 

We distinguish two cases. First, assume $c \le \tfrac 14$.
To obtain a lower bound for $|\cS(n,c)|$ in this case, we simply count
bipartite graphs with density~$c$ and with colour classes of size $n/2$.
There are 
\[
	\binom{\frac{n^2}4}{c\binom n2} \ge \binom{\frac12\binom n2}{c\binom n2} \ge
	2^{\frac 12 \binom n2 H(2c) - \gamma \binom n2}
\]
such graphs. 

Now assume that $\tfrac14<c\le\tfrac12$. In this case we construct suitable
$k$-partite graphs. For this purpose choose~$k$ such that
\begin{equation}\label{eq:lower}
	x:= c\binom n2-(k-2)(\tfrac n2-k+2)-\binom{k-2}{2} \in
	\Big[\frac{\,n^2}8-n,\frac{\,n^2}8+n\Big] \,.
\end{equation}
Now, construct~$k$ (independent) vertex sets $V_1,\dots,V_k$ with $|V_1| = \tfrac n2$,
$|V_2|=\tfrac n2-k+2$ and $|V_i|=1$ for $i=\{3,\dots,k\}$ and insert all
edges between $V_i$ and $V_j$ with $i,j \in[k]\setminus \{1\}$, $i\neq j$.
Call the resulting graph~$G_0$. By~\eqref{eq:lower} we obtain a generalised
split graph with density~$c$ from~$G_0$, if we insert~$x$ edges between~$V_1$
and~$V_2\cup\dots\cup V_k$. Since this can be done in at least
\[
    \binom{\tfrac{n^2}4}{\tfrac{n^2}{8}-n} \ge 2^{\tfrac 12 \binom n2 - \gamma
      \binom n2}
\]
ways, we obtain at least $2^{\tfrac 12 \binom n2 - \gamma\binom n2}$
generalised split graphs with exactly $c\binom n2$ edges.
\end{proof}

\subsection{Regularity}
\label{sec:reg} 

In order to prove the upper bound from Theorem~\ref{thm:main}, i.e.,
$$|\cC(n,c)|\le 2^{h(c)\binom n2 + \gamma \binom n2},$$ we will analyse the
structure of graphs in $\cC(n,c)$ by applying a variant of the regularity lemma suitable for our purposes.

Let $G=(V,E)$ be a graph. For disjoint nonempty vertex sets $A,B\subset V$ the
\emph{density} $d(A,B):=e(A,B)/(|A||B|)$ of the pair $(A,B)$ is the number of
edges that run between $A$ and $B$ divided by the number of possible edges
between $A$ and $B$. In the following let $\eps,d\in[0,1]$.
The pair $(A,B)$ is \emph{$\eps$-regular}, if for all $A' \subset A$ and $B'
\subset B$ with $|A'|\geq\eps|A|$ and $|B'|\geq\eps|B|$ it is true that
$|d(A,B)-d(A',B')|\le\eps$. An $\eps$-regular pair $(A,B)$ is called
\emph{$(\eps,d)$-regular}, if it has density at least~$d$.

A partition
$V_0\dcup V_1 \dcup\dots\dcup V_k$ of $V$ is an \emph{equipartition} if
$|V_i|=|V_j|$ for all $i,j\in[k]$. 
An \emph{($\eps,d$)-regular partition} of~$G$ with \emph{reduced graph}
$R=(V_R,E_R)$ is an equipartition $V_0\dcup V_1 \dcup\dots\dcup V_k$ of $V$ with
$|V_0|\le\eps|V|$, and $V_R=[k]$ such that $(V_i,V_j)$ is an $(\eps,d)$-regular
pair in $G$ iff $\{i,j\}\in E_R$. In this case we also call~$R$ an
$(\eps,d)$-reduced graph of~$G$.
An $(\eps,0)$-regular partition $V_0\dcup V_1 \dcup\dots\dcup V_k$
which has at most $\eps\binom{k}{2}$ pairs that are not $\eps$-regular is
also called \emph{$\eps$-regular partition}.
The partition classes $V_i$ with $i\in[k]$ are called \emph{clusters} of
$G$ and $V_0$ is the \emph{exceptional set}.  

With this terminology at hand we can state the celebrated regularity lemma
of Szemer\'edi.

\begin{lemma}[regularity lemma,~\cite{Szem}]\label{reg:lem:reg}
  For all $\eps>0$ and $k_0$ there is $k_1$ such that
  every graph $G=(V,E)$ on $n\ge k_1$ vertices has an 
  $\eps$-regular partition
  $V=V_0\dcup V_1\dcup\dots\dcup V_k$ with 
  $k_0\le k\le k_1$.
\qed
\end{lemma}

The strength of this lemma becomes apparent when it is complemented with
corresponding embedding lemmas, such as the following (see, e.g.,
\cite{RegularityLemmaApplications}). A \emph{homomorphism} from a graph
$H=(V_H,E_H)$ to a graph $R=(V_R,E_R)$ is an edge-preserving mapping from
$V_H$ to $V_R$.

\begin{lemma}[embedding lemma]
\label{reg:lem:emb}
For every $d>0$ and every integer $k$ there exists $\eps>0$ 
with the following property. Let $H$ be a graph on $k$ vertices
$v_1,\dots,v_k$. Let~$G$ be a graph.  Let
$V_1,\dots,V_k$ be clusters of an $(\eps,d)$-regular partition of~$G$ with
reduced graph $R = ([k],E_R)$. If there is a homomorphism from $H$ to $R$,
then $H$ is a subgraph of~$G$.  \qed
\end{lemma}

\subsection{Regular partitions and counting} 
\label{sec:subgraph}


As a warm up (and for the sake of completeness) we consider the problem of
counting graphs of a fixed density without a given (not necessarily
induced) subgraph~$F$ and prove~\eqref{eq:subgraph}.  For this purpose we
mimic the proof given by Erd\H{o}s, Frankl and R\"odl in \cite{ErdFraRod}
for the corresponding problem without fixed density.

\begin{proof}[Proof of~\eqref{eq:subgraph}]
  Let $F$ be a graph with $\chi(F)=r$ and $c>0$. Let $\gamma>0$ be
  given. For large enough $n$ the lower bound
  \begin{align*}
    |\Forb_n(F,c)|\ge 2^{\frac{r-2}{r-1}H(\frac{r-1}{r-2}c)\binom n2-\gamma
      \binom n2}
  \end{align*}
  can easily be obtained by counting subgraphs with $c\binom n2$ edges of
  the complete ($r-1$)-partite graph with $n/r$ vertices in each part, and by
  applying~\eqref{eq:BinCoe}.

  It remains to show the upper bound
  \begin{equation*}
    |\Forb_n(F,c)|\le 2^{\frac{r-2}{r-1}H(\frac{r-1}{r-2}c)\binom n2 +
      \gamma \binom n2}
  \end{equation*}
  for $n$ sufficiently large.  We choose~$d$ such that 
  \begin{equation*}
    0<d\le \min\Big\{\frac
    1{16}\,,\, \frac14\gamma^2 \,,\,
    \Big(-2\frac{r-1}{r-2}\log\big(1-c\frac{r-1}{r-2}\big)\Big)^{-2}
    \Big\} \,.
  \end{equation*}
  Since the entropy function is concave we have for each $\hat d\in[0,d]$ that
  \begin{equation}\label{eq:Defd}
    \begin{aligned}
      H\big((c-2\hat d\big)\tfrac{r-1}{r-2}) 
      &\le H(c\tfrac{r-1}{r-2})-2\hat d\tfrac{r-1}{r-2}
              H'(c\tfrac{r-1}{r-2}) \\
      &\le 
           H(c\tfrac{r-1}{r-2})
              -2\hat d\tfrac{r-1}{r-2}\log(1-c\tfrac{r-1}{r-2})
        \le H(c\tfrac{r-1}{r-2}) +\sqrt{\hat d} \,,
    \end{aligned}
  \end{equation}
  which we shall use later.
  Next, let~$\eps$ be the constant returned from
  Lemma~\ref{reg:lem:emb} for input~$d$ and with~$k$ replaced by~$r$. 
  Set $k_0=\lceil 10/d\rceil$ 
  and let~$k_1$ be the constant returned by
 Lemma~\ref{reg:lem:reg} for input~$k_0$ and~$\eps$.
  Further let $n\ge k_1$.

  Now we use the regularity lemma, Lemma~\ref{reg:lem:reg}, with parameters
  $\eps$, $k_0$ for each graph~$G$ in $\Forb_n(F,c)$. For each
  such application the regularity lemma produces an $\eps$-regular
  partition with at most~$k_1$ clusters, for which we can construct the
  corresponding $(\eps,d)$-reduced graph~$R$. Since~$k_1$ is finite there
  is only a finite number of different reduced graphs~$R$ resulting from
  these applications of the regularity lemma. Hence we can partition
  $\Forb_n(F,c)$ into a finite number of classes $\cR(R,\eps,d,n,F,c)$ of
  graphs with $(\eps,d)$-reduced graph $R$.
  Accordingly, it suffices to show that for each~$R$ we have
  \begin{equation}\label{eq:subgraph:goal}
      |\cR(R,\eps,d,n,F,c)|\le 2^{\frac{r-2}{r-1}H(\frac{r-1}{r-2}c)\binom n2 +
      \gamma \binom n2} \,.
  \end{equation}

  Let $R=(V_R,E_R)$ be any graph such that $\cR(R,\eps,d,n,F,c)$ is
  non-empty, let $k=|V_R|$,  let $G\in\cR(R,\eps,d,n,F,c)$ and let $P$ be an $\eps$-regular partition
  of~$G$ corresponding to~$R$.
  By the choice of~$k_0$ at most $k\binom{n/k}{2}\le\tfrac d2 \binom n2$ edges of~$G$
  are inside clusters of~$P$, at most $d\binom{n}{2}$ edges of $G$ are in
  regular pairs of~$P$ with density less than $d$, and at most $2\eps
  n^2\le \tfrac d2 \binom n2$ edges of~$G$ are in irregular pairs of~$P$ or
  have a vertex in the exceptional set. We conclude that at least
  $(c-2d)\binom n2$ edges of~$G$ lie in $(\eps,d)$-regular
  pairs of~$P$.
  In addition, by the choice of~$\eps$
  and since~$F$ has chromatic number~$r$, Lemma~\ref{reg:lem:emb}
  implies that $K_r\not\subseteq R$. It follows from Tur\'an's Theorem that
  $|E_R| \le \tfrac{r-2}{r-1}\binom{k}{2}$. 
  Summarising, we can bound the number of graphs in $\cR(R,\eps,d,n,F,c)$
  by bounding the number of ways to distribute at least
  $(c-2d)\binom n2$ edges to at most $\tfrac{r-2}{r-1}\binom{k}{2}$
  regular pairs (corresponding to edges of~$R$) with clusters of size at
  most $n/k$, and distributing at most $2d\binom{n}{2}$ edges arbitrarily.
  By the choice of~$n$, the first of these two factors can be bounded by
  \begin{equation*}
  \begin{split}
    \max\limits_{0\le\hat d \le d}
   \binom{\frac{r-2}{r-1}\binom{n}{2}}{(c-2\hat d)\binom n2} 
    &\leByRef{eq:BinCoe}
    \max\limits_{0\le\hat d \le d}
   2^{\frac{r-2}{r-1}\binom{n}{2} H\left(\frac{r-1}{r-2}(c-2\hat d) \right)} \\
    &\leByRef{eq:Defd}
    \max\limits_{0\le\hat d \le d}
    2^{\frac{r-2}{r-1}\binom n2 H\left(c\frac{r-1}{r-2}\right)+
     \sqrt{\hat d}\binom n2} \\
   &\le
    2^{\frac{r-2}{r-1}\binom n2 H\left(c\frac{r-1}{r-2}\right)+
     \sqrt{d}\binom n2}
  \,,
  \end{split}
  \end{equation*}
  and the second by $\le 2^{2d \binom n2}$. Since
  $2d +\sqrt{d}\le \gamma$ this implies~\eqref{eq:subgraph:goal} as desired.
\end{proof}

\subsection{Embedding induced subgraphs}
\label{sec:emb}

In the last section we showed how the regularity lemma and a corresponding
embedding lemma can be used to count graphs with forbidden subgraphs. 
In this section we provide the tools that will allow us to adapt this
strategy to the setting of forbidden induced subgraphs.

We remark that the concepts and ideas presented in this section are not new.
They were used for various similar applications, e.g., by Bollob\'as and
Thomason~\cite{BolTho00} or Loebl, Reed, Scott, Thomason, and
Thomass\'e~\cite{LoeReeScoThoTho}, 
as well as for different applications such as property testing,
e.g., by Alon, Fischer, Krivelevich and Szegedy~\cite{AloFisKriSze}, or Alon and
Shapira \cite{AloSha}.

\smallskip

We start with an embedding lemma for induced subgraphs, which allows us to
find an induced copy of a graph~$F$ in a graph~$G$ with reduced graph~$R$
if~$F$ is an induced subgraph of~$R$ (see, e.g., \cite{AloFisKriSze}).

\begin{lemma}[injective embedding lemma for induced subgraphs]
  \label{lem:EmbedInd} For every $d>0$ and every integer $k$ there exists
  $\eps>0$ such that for all $f\le k$ the following holds. Let $V_1,\dots,V_f$ be clusters of
  an $\eps$-regular partition of a graph $G$ such that for all $1\le i<j\le
  f$ the pair $(V_i,V_j)$ is $\eps$-regular. Let $F=(V_F,E_F)$ be a graph
  on $f$ vertices and let $g\colon V_F\to[f]$ be an injective mapping
  from~$F$ to the clusters of~$G$ such that for all $1\le i<j\le
  f$ we have $d(V_i,V_j)\ge d$ if
  $\{g^{-1}(i),g^{-1}(j)\} \in E(F)$ and $d(V_i,V_j)\le 1-d$ if
  $\{g^{-1}(i),g^{-1}(j)\} \not\in E(F)$.  Then $G$ contains an
  induced copy of $F$.  \qed
\end{lemma}

In contrast to Lemma~\ref{reg:lem:emb}, this lemma allows us only to embed
one vertex per cluster of~$G$. Our goal in the following will be to
describe an embedding lemma for induced subgraphs which allows us to embed
more than one vertex per cluster.  Observe first, that for this purpose we
must have some control over the existence of edges respectively non-edges
\emph{inside} clusters of a regular partition of~$G$. This can be achieved
by applying the following lemma to each of these clusters. It is not
difficult to infer this lemma from the regularity lemma
(Lemma~\ref{reg:lem:reg}) by applying Tur\'an's theorem and Ramsey's
theorem (see, e.g., \cite{AloFisKriSze}).

We use the following definition. An \emph{$(\mu,\eps,k)$-subpartition}
of a graph $G=(V,E)$ is a family of pairwise disjoint vertex sets
$W_1,\dots,W_k \subseteq V$ with $|W_i| \ge \mu |V|$ for all
$i\in[k]$ such that each pair $(W_i,W_j)$ with $\{i,j\}\in\binom k2$ is
$\eps$-regular.  A $(\mu,\eps,k)$-subpartition $W_1,\dots,W_k$
of~$G$ is \emph{dense} if $d(W_i,W_j)\ge\frac12$ for all $\{i,j\}\in\binom
k2$, and \emph{sparse} if $d(W_i,W_j)<\frac12$ for all $\{i,j\}\in\binom
k2$.

\begin{lemma} \label{lem:VertexColour} For every $k$ and $\eps$ there
  exists $\mu>0$ such that every graph $G=(V,E)$
  with $n\ge \mu^{-1}$ vertices either has a sparse or a dense
  $(\mu,\eps,k)$-subpartition.  \qed
\end{lemma}

The idea for the embedding lemma for induced subgraphs~$F$ of~$G$ now is as
follows.  We first find a regular partition of~$G$. By
Lemma~\ref{lem:EmbedInd}, if a regular pair~$(V_i,V_j)$ in this partition
is very dense then we can embed edges of~$F$ into $(V_i,V_j)$, if it
is very sparse then we can embed non-edges of~$F$, and if its density is
neither very small nor very big then we can embed both edges and non-edges
of~$F$. Moreover, Lemma~\ref{lem:VertexColour} asserts that each cluster
either has a sparse or a dense subpartition.  In the first case we can
embed non-edges inside this cluster, in the second case we can embed edges.

This motivates that we want to tag the reduced graphs with some additional
information. For this purpose we colour an edge of the reduced graph white
if the corresponding regular pair is sparse, grey if it is of medium
density, and black if it is dense. Moreover we colour a cluster white if
it has a sparse subpartition and black otherwise. We call a cluster graph that is
coloured in this way a type. The following definitions make this
precise.
%

\begin{definition}[coloured graph, type]
\label{def:type}
  A \emph{coloured graph} $R$ is a triple $(V_R,E_R,\sigma)$ such that
  $(V_R,E_R)$ is a graph and $\sigma\colon V_R\cup E_R\to\{0,\frac12,1\}$ is a
  colouring of the vertices and the edges of this graph where
  $\sigma(V_R)\subset\{0,1\}$. Vertices and edges with colour $0$, $\frac12$, 
  and $1$ are also called \emph{white}, \emph{grey}, and \emph{black},
  respectively.

  Let $G=(V,E)$ be a graph and let $V=V_0\dcup V_1\dcup\dots\dcup V_k$ be an
  $\eps$-regular partition of~$G$ with reduced graph $([k],E_R)$.  The
  \emph{$(\eps,\eps',d,k')$-type} $R$ corresponding to the partition $V_0\dcup
  V_1\dcup\dots\dcup V_k$ is the coloured graph $R=([k],E_R,\sigma)$ with colouring
  \begin{equation*}
    \sigma(\{i,j\}) = 
    \begin{cases}
      0 & \text{ if $d(V_i,V_j)<d$}\,, \\
      1 & \text{ if $d(V_i,V_j)>1-d$}\,, \\
      \tfrac12 & \text{ otherwise}\,,
    \end{cases}
  \end{equation*}
  for all $\{i,j\}\in E_R$ and
 \begin{align*}
    \sigma(i) =
    \begin{cases}
      0 & \text{ if $G[V_i]$ has a sparse $(\mu,\eps',k')$-subpartition}\,, \\
      1 & \text{ if $G[V_i]$ has a dense $(\mu,\eps',k')$-subpartition}\,,
    \end{cases}
  \end{align*}
  for all $i\in[k]$, where $\mu$ is the constant from
  Lemma~\ref{lem:VertexColour} for input~$k'$ and~$\eps'$.
  In this case we also simply say that~$G$ has $(\eps,\eps',d,k')$-type~$R$.
\end{definition}

By the discussion above a combination of the regularity lemma,
Lemma~\ref{reg:lem:reg}, and Lemma~\ref{lem:VertexColour} gives the following.

\begin{lemma}[type lemma]
\label{lem:type}
  For every $\eps,\eps'\in(0,\frac12)$ and for all integers~$k'$, $k_0$ there
  are integers $k_1$ and $n_0$ such that for every $d>0$
  every graph $G$ on at least $n_0$ vertices has an
  $(\eps,\eps',d,k')$-type $R=([k],E_R,\sigma)$ with $k_0\le k\le k_1$ and with
  at most $\eps k^2$ non-edges.
\end{lemma}
\begin{proof}
  Given $\eps$,~$\eps'$ and~$k'$,~$k_0$ we let~$k_1$ be the constant returned
  from Lemma~\ref{reg:lem:reg} for input~$\eps$ and~$k_0$, and~$\mu$ be the
  constant returned from Lemma~\ref{lem:VertexColour} for input~$\eps'$
  and~$k'$. Set $n_0:=2\mu^{-1}k_1$.

  Now let $d$ be given and let $G$ be a graph on at least $n_0$ vertices.  By
  Lemma~\ref{reg:lem:reg} the graph~$G$ has an $\eps$-regular partition
  $V=V_0\dcup V_1 \dcup \dots \dcup V_k$ such that $k_0\le k\le k_1$. By
  definition at most $\eps\binom{k}{2}\le\eps k^2$ pairs $(V_i,V_j)$ are not
  $\eps$-regular. Let $R'=([k],E_R)$ be the $\eps$-reduced graph of $V_0\dcup
  V_1 \dcup \dots \dcup V_k$. It follows that~$R'$ has at most $\eps k^2$
  non-edges.  Let $i\in[k]$. Since $|V_i|\ge(1-\eps)n/k\ge
  n/(2k_1)\ge\mu^{-1}$ we can apply Lemma~\ref{lem:VertexColour} and conclude
  that $V_i$ either has a sparse or a dense
  $(\mu,\eps',k')$-subpartition. Accordingly we obtain an
  $(\eps,\eps',d,k')$-type $R=([k],E_R,\sigma)$ for~$G$ by colouring the edges
  and vertices of~$R'$ as specified in Definition~\ref{def:type}.
\end{proof}

For formulating our embedding lemma we need one last preparation. We 
generalise the concept of a graph homomorphism to the setting of
coloured graphs.

\begin{definition}[coloured homomorphism]
\label{def:colhom}
 Let~$F=(V_F,E_F)$ be a graph and $R=(V_R,E_R,\sigma)$ be a coloured graph.
  A \emph{coloured homomorphism} from~$F$ to~$R$ is a mapping $h\colon V_F \to
  V_R$ with the following properties.
  \begin{enumerate}[label=\abc]
  \item If $u,v\in V_F$ and $h(u)\neq h(v)$ then $\{h(u),h(v)\}\in E_R$.
  \item If $\{u,v\}\in E_F$ then \\ $h(u)=h(v)$ and $\sigma(h(u))=1$, or
    $h(u)\neq h(v)$ and
    $\sigma\Big(\big\{h(u),h(v)\big\}\Big)\in\{\tfrac12,1\}$.
    \item If $\{u,v\}\notin E_F$ then \\ $h(u)=h(v)$ and $\sigma(h(u))=0$,
    or $h(u)\neq h(v)$ and $\sigma\Big(\big\{h(u),h(v)\big\}\Big)\in\{0,\tfrac12\}$.
  \end{enumerate}
  If there is a coloured homomorphism from~$F$ to~$R$ we also write $F\colhom
  R$.
\end{definition}

The following embedding lemma states that a graph $F$ is an induced
subgraph of a graph $G$ with type $R$ if there is a coloured homomorphism
from~$F$ to~$R$.  This lemma is, e.g, inherent in~\cite{AloSha}. For
completeness we provide its proof below.

\begin{lemma}[embedding lemma for induced graphs] \label{lem:embed}
  For every pair of integers~$k$, $k'$ and for every~$d\in(0,1)$ there are
  $\eps,\eps'>0$ such that the following holds.
  Let $f\le k'$ and~$G$ be a graph on~$n$ vertices with
  $(\eps,\eps',d,k')$-type $R$ on~$k$ vertices. Let~$F$ be an $f$-vertex graph such that there is a coloured
  homomorphism from $F$ to $R$. Then~$F$ is an induced subgraph of~$G$.
\end{lemma}

\begin{proof}
  Let $k,k'\in\N$ and $d\in(0,1)$ be given. Let~$\eps'$ be given by
  Lemma~\ref{lem:EmbedInd} for input $\tfrac d2$ and~$k'$.
  Let $\mu$ be the constant from Lemma~\ref{lem:VertexColour} for input~$k'$
  and~$\eps'$.  Set $\eps:=\min\{ d/2, \mu\eps' \}$.

  Let $G$, $R=([k],E_R,\sigma)$ and~$F$ be as required, let $V_0\dcup
  V_1\dcup\dots\dcup V_k$ be an $\eps$-regular partition of~$G$
  corresponding to~$R$ and let $h\colon F \colhom R$ be a coloured
  homomorphism from~$F$ to~$R$.  For each $i\in[k]$ we have by definition, if
  $\sigma(V_i)=0$ then $V_i$ has a sparse $(\mu,\eps',k')$-subpartition
  $W_{i,1},\dots,W_{i,k'}$, and if $\sigma(V_i)=1$ then $V_i$ has a dense
  $(\mu,\eps',k')$-subpartition $W_{i,1},\dots,W_{i,k'}$.

  Observe that $\bigcup_{i\in[k],j\in[k']} W_{i,j}$ has the following
  properties, since $\eps'\ge\eps/\mu$ and $\eps\le d/2$.
  If $\{i,i'\}\in E_R$ then for all $j,j'\in[k']$ the pair
 \begin{equation}\label{eq:embed:eps1}
   \text{$(W_{i,j},W_{i',j'})$ is $\eps'$-regular,}
 \end{equation}
 and has density
  \begin{equation}\label{eq:embed:d1}
   d(W_{i,j},W_{i',j'}) \in
   \begin{cases}
     [0,2d) & \text{if $\sigma(\{i,i'\})=0$} \,,\\
     (1-2d,1] & \text{if $\sigma(\{i,i'\})=1$} \,,\\
     (\frac d2, 1-\frac d2)  & \text{if $\sigma(\{i,i'\})=\frac12$}
     \,.
  \end{cases}
 \end{equation}
 Moreover, for all $i\in[k]$ and all $j,j'\in[k']$ with $j\neq j'$ the pair
 \begin{equation}\label{eq:embed:eps2}
   \text{
     $(W_{i,j},W_{i,j'})$ is $\eps'$-regular}
\end{equation}
 and has density
 \begin{equation}\label{eq:embed:d2}
   d(W_{i,j},W_{i,j'}) \in
   \begin{cases}
     [0,\frac12) & \text{if $\sigma(i)=0$} \,, \\
     [\frac12,1]  & \text{if $\sigma(i)=1$}
    \,.
  \end{cases}
 \end{equation}

  Now, we define an injective mapping $g\colon V(F) \to [k]\times[k']$ as
  follows. For~$i\in[k]$ let $F_i:=\{x\in V(F)\colon h(x)=i \}$ and name the
  vertices in $F_i$ arbitrarily by
  $\{x_{i,1},\dots,x_{i,f_i}\}=F_i$. Set $g(x_{i,j}):=(i,j)$ for all $j\in[f_i]$.
  This is well-defined since $|F_i|=f_i\le f\le k'$. 
  Let~$I\subset[k]$ with $|I|\le f$ be the set of indices $i\in[k]$ such that
  $F_i\neq\emptyset$. 

  We claim that~$G$, $\bigcup_{i\in I,j\in[f_i]}
  W_{i,j}$, $F$, and $g$ satisfy the conditions of Lemma~\ref{lem:EmbedInd}
  with parameters~$d/2$, $k'$, $\eps'$, and~$f$. Indeed,
  by~\eqref{eq:embed:eps2} each cluster pair $(W_{i,j},W_{i,j'})$ with
  $i\in I$ and $j,j'\in[k']$, $j\neq j' $ is $\eps'$-regular. Moreover, for
  each $i,i'\in I$ with $i\neq i'$ we have that there are $x\in F_i$ and
  $y\in F_{i'}$. By the definition of a coloured homomorphism
  (Definition~\ref{def:colhom}) we have that
  $\{V_i,V_{i'}\}=\{h(x),h(y)\}\in E_R$. Hence~\eqref{eq:embed:eps1}
  implies that also $(W_{i,j},W_{i',j'})$ is $\eps'$-regular.  It remains
  to show that, if $x,y$ are two vertices in $V(F)$ and $(i,j)=g(x)$ and
  $(i',j')=g(y)$ then $d(W_{i,j},W_{i',j'})\ge d/2$ if $\{x,y\}\in E(F)$
  and $d(W_{i,j},W_{i',j'})\le 1-d/2$ otherwise. To see this, assume first
  that $\{x,y\}\in E(F)$. Then, by the definition of a coloured
  homomorphism, either $h(x)=h(y)$ and $\sigma(h(x))=1$, which implies
  $d(W_{i,j},W_{i',j'})\ge\frac12\ge d/2$ by~\eqref{eq:embed:d2}. Or $h(x)\neq
  h(y)$ and $\sigma(\{h(x),h(y)\})\ge\frac12$ and hence we have
  $d(W_{i,j},W_{i',j'})\ge d/2$ by~\eqref{eq:embed:d1}. If $\{x,y\}\not\in
  E(F)$ on the other hand, then either $h(x)=h(y)$ and $\sigma(h(x))=0$,
  and so $d(W_{i,j},W_{i',j'})\le\frac12\le 1- d/2$ by~\eqref{eq:embed:d2}. Or
  $h(x)\neq h(y)$ and $\sigma(\{h(x),h(y)\})\le\frac12$ and thus
  $d(W_{i,j},W_{i',j'})\le 1-d/2$ by~\eqref{eq:embed:d1}.
  
  It follows that we can indeed apply Lemma~\ref{lem:EmbedInd} and conclude
  that~$F$ is an induced subgraph of~$G$ as desired.
\end{proof}

\subsection{The upper bound of Theorem~\ref{thm:main}}
\label{sec:upper}

Now we are ready to prove the upper bound of Theorem~\ref{thm:main}, that is, we
establish the following lemma.

\begin{lemma}
\label{lem:upper}
 For all $c,\gamma\in(0,1)$ there is~$n_0$ such that for all
 $n\ge n_0$ we have \[|\cC(n,c)|\le 2^{h(c)\binom{n}{2}+\gamma \binom n2}\,.\]
\end{lemma}

The idea of the proof of Lemma~\ref{lem:upper} is as follows. We proceed in
three steps.
Firstly, similarly as in
  the proof of~\eqref{eq:subgraph} in Section~\ref{sec:subgraph} we start by
  applying the regularity lemma to all graphs in $\cC(n,c)$. For each of the
  regular partitions obtained in this way there is a corresponding type, and in
  total we only get a constant number~$K$ of different types. 
Secondly, we continue with a
  structural analysis of the possible types~$R$ for graphs from $\cC(n,c)$ and
  infer from Lemma~\ref{lem:embed}
  that~$R$ cannot contain a triangle all of whose edges are grey (see
  Lemma~\ref{lem:GreyTriangle}). 
Thirdly, we  prove that a coloured graph without
  such a grey triangle can only serve as a type for at most $\UB(n)$ graphs
  on~$n$ vertices (see Lemma~\ref{lem:Counting}). Multiplying $\UB(n)$ with~$K$
  then gives the desired bound.

We start with the second step.

\begin{lemma}\label{lem:GreyTriangle}
  For every integer $k'\ge 5$ and every $d>0$ there exist
  $\eps_{\subref{lem:GreyTriangle}},\eps'_{\subref{lem:GreyTriangle}}>0$ such
  that for every $0<\eps\le\eps_{\subref{lem:GreyTriangle}}$ and every
  $0<\eps'\le\eps'_{\subref{lem:GreyTriangle}}$ the following is true. If~$G$ is
  a graph whose ($\eps,\eps',d,k'$)-type~$R$ contains three grey edges forming a
  triangle then $G$ has an induced~$C_5$.
\end{lemma}
\begin{proof}
  Given $d>0$, set $k:=3$ and $k':=5$, and let
  $\eps_{\subref{lem:GreyTriangle}},\eps'_{\subref{lem:GreyTriangle}}>0$ be the
  constants given by Lemma~\ref{lem:embed} for~$d$, $k$, and~$k'$. Let positive
  constants $\eps\le\eps_{\subref{lem:GreyTriangle}}$ and
  $\eps'\le\eps'_{\subref{lem:GreyTriangle}}$ be given. Let $G$ be a graph
  with ($\eps,\eps',d,k'$)-type~$R$ such that~$R$ contains a
  triangle~$T$ with three grey edges.

  By Lemma~\ref{lem:embed} the graph~$G$ contains an induced~$C_5$ if there
  exists a coloured homomorphism from~$C_5$ to~$T$, in which case we are
  done. We claim that such a coloured homomorphism~$h$ does always exist
  (regardless of the colours of the vertices of~$T$). Indeed, if~$T$ has at
  least two black vertices~$V_1$, $V_2$ then we can construct~$h$ by mapping a
  pair of adjacent vertices of~$C_5$ to~$V_1$, a disjoint pair of adjacent
  vertices of~$C_5$ to~$V_2$, and the remaining vertex of~$C_5$ to the remaining
  vertex of~$T$.  If~$T$ has at least two white vertices~$V_1$, $V_2$, on the
  other hand, then we can construct~$h$ by mapping one pair of non-adjacent
  vertices of~$C_5$ to~$V_1$, a disjoint pair of non-adjacent vertices of~$C_5$
  to~$V_2$ and the remaining vertex of~$C_5$ to the remaining vertex of~$T$.
\end{proof}

Next, we show an upper bound on the number of graphs on~$n$ vertices with a
fixed type~$R$, where~$R$ does not contain a triangle with three grey edges.
We use the following definition.
\begin{equation}\label{eq:cR}
  \cR(R,\eps,\eps',d,k',n,c) := \big\{ G \in \cG(n,c)\colon 
        \text{ $G$ has $(\eps,\eps',d,k')$-type~$R$}\big\}. 
\end{equation}
We stress that $\cR(R,\eps,\eps',d,k',n,c)$ and $\cR(R',\eps,\eps',d,k',n,c)$
may have non-empty intersection for $R\neq R'$.

\begin{lemma} \label{lem:Counting} 
  For every $c$ with $0<c\le\frac12$, and every $\gamma>0$ there exist
  $\eps_{\subref{lem:Counting}},d_0>0$ and integers $n_{\subref{lem:Counting}}$,
  $k_0$ such that for all positive $d\le d_0$, $\eps\le
  \eps_{\subref{lem:Counting}}$, $\eps'$, and all integers $n\ge
  n_{\subref{lem:Counting}}$, $k\ge k_0$, $k'$ the following holds. If $R$ is a
  coloured graph of order $k$ which has at most $\eps k^2$ non-edges and does not
  contain a triangle with three grey edges, then
  \[
	|\cR(R,\eps,\eps',d,k',n,c)| \le 2^{h(c)\binom n2 + \gamma \binom n2}\; .
  \]
\end{lemma}
\begin{proof}
  Let $c, \gamma$ be given. Choose $\eps_{\subref{lem:Counting}}$, $d_0$, $k_0$
  such that $\max\{4\eps_{\subref{lem:Counting}}, H(d_0), \tfrac{1}{k_0}\} \le
  \tfrac{\gamma}{5}$. Let $n_{\subref{lem:Counting}}$ be large enough to
  guarantee $\log(n+1) \le \tfrac{\gamma}{5} (n-1)$. Let
  $\eps\le\eps_{\subref{lem:Counting}}$, $\eps'$, $d\le d_0$, $n\ge n_0$, $k\ge
  k_0$, $k'$ be given.

  Let~$R=([k],E_R,\sigma)$ be a coloured graph which has at most $\eps k^2$
  non-edges and does not contain a triangle with three grey edges.
  We shall count the number of graphs in $\cR(R,\eps,\eps',d,k',n,c)$ by
  estimating the number of equipartitions $V_0\dcup\dots\dcup V_k$ of $[n]$,
  the number of choices for edges with one end in the exceptional set~$V_0$ and
  edges in pairs $(V_i,V_j)$ such that $\{i,j\}\not\in E_R$, 
  the number of choices for edges in clusters~$V_i$ such that $i$ is white or
  black in~$R$,
  and the number of choices for at most $c\binom n2$ edges
  in pairs $(V_i,V_j)$ such that $\{i,j\}$ is a white, black, or grey edge of~$R$.

  The number of equipartitions $V_0\dcup\dots\dcup V_k$ of~$[n]$ is bounded by
  \begin{equation} \label{eq:Count0}
    (k+1)^n = 2^{n \log (k+1)} \le 2^{\frac{\gamma}5\binom n2}. 
  \end{equation}
  Let us now fix such an equipartition.
  There are at most $\eps n^2$ possible edges that have at
  least one end in $V_0$ and at most 
  $\tfrac{\eps}{2} n^2$ possible edges in pairs $(V_i,V_j)$ such that
  $\{i,j\}\not\in E_R$. Thus there are at most
  \begin{equation} \label{eq:Count1} 
    2^{\frac 32 \eps n^2 } \le 2^{4\eps\binom n2}
  \end{equation}
  possible ways to distribute such edges.  In addition, the number of ways to
  distribute edges in clusters $V_i$ corresponding to white or black vertices of
  $R$ is at most
  \begin{equation} \label{eq:Count3}
    2^{k\binom{n/k}{2}} \le 2^{\frac 1k \binom n2}
  \end{equation}
  By definition, white edges of an ($\eps,\eps',d,k'$)-type correspond to pairs
  with density at most~$d$ and black edges correspond to pairs with density at
  least $(1-d)$. Hence, by the symmetry of the binomial coefficient the number of
  ways to distribute edges in pairs $(V_i,V_j)$ such that $\{i,j\}$ is a white or
  a black edge of~$R$ is at most
  \begin{equation} \label{eq:Count2}
    \binom{\left(\frac nk\right)^2}{d\left(\frac nk\right)^2}^{\binom k2} \le
    2^{\binom k2 \left(\frac nk\right)^2 H(d)} \le 2^{H(d)\binom n2} \,.
  \end{equation}
  For later reference we now sum up the estimates obtained so far. The
  product of \eqref{eq:Count0}--\eqref{eq:Count2} gives less than
  \begin{equation}\label{eq:count:prod}
    2^{\bleft( \frac{\gamma}{5}  + 4\eps + H(d) + \frac 1k \right) \binom n2}
    \le 2^{\frac45\gamma\binom n2}
  \end{equation}
  choices for the partition $V_0\dcup\dots\dcup V_k$ and for the
  distribution of edges inside such a partition, besides to the pairs
  $(V_i,V_j)$ corresponding to grey edges of~$R$.

  It remains to take the grey edges~$E_g$ of~$R$ into account.
  By assumption $E_g$ does not contain a triangle.
  Hence, by Tur\'an's Theorem (see, e.g.,
  \cite{ExtremalProblemGraphTheory}) we have
  $|E_g|\le\tfrac{k^2}{4}$. It follows that there are at most
  $\tfrac{k^2}{4}\left(\tfrac nk\right)^2 = \tfrac{n^2}{4}$ possible places for
  edges in $E_g$-pairs $(V_i,V_j)$, i.e., pairs such that $\{i,j\}\in E_g$.
  Hence the number~$N_g$ of possible ways to distribute at most $c\binom n2$
  edges to $E_g$-pairs is at most $\binom{n^2/4}{c\binom n2}$.
  If $c<\tfrac14$ then this gives
  \begin{equation} \label{eq:Count4}
     N_g\le 2^{\frac 12 \binom{n}{2} H(2c) +\gamma\binom n2} \,,
  \end{equation}
  and if $\tfrac 14 \le c \le \tfrac 12$ then 
  \begin{equation} \label{eq:Count5}
    N_g \le 2^{\frac{n^2}4} \le 2^{\frac 12 \binom{n}{2} +
      \frac{\gamma}{5} \binom n2}\;.
  \end{equation}
  Combining~\eqref{eq:Count4} and \eqref{eq:Count5} and recalling the definition
  of $h(c)$ in~\eqref{eq:h} gives
  \begin{equation} \label{eq:Count6}
    N_g \le 2^{h(c)\binom{n}{2} + \frac{\gamma}{5} \binom n2}\; .
  \end{equation}
  Multiplying~\eqref{eq:count:prod}
  and~\eqref{eq:Count6} gives the desired upper bound
  \begin{equation*}
    |\cR(R,\eps,\eps',d,k',n,c)| \le 2^{h(c)\binom n2 + \frac{\gamma}{5}
      \binom n2 + \frac{4}{5}\gamma \binom n2} 
    = 2^{h(c)\binom n2 + \gamma \binom n2}\; .
  \end{equation*}
\end{proof}

With this we are in position to prove Lemma~\ref{lem:upper}.

\begin{proof}[Proof of Lemma~\ref{lem:upper}]
  Observe first that it suffices to prove the lemma for $c\le\frac12$, since the
  complement of a graph without induced~$C_5$ is induced $C_5$-free and hence
  $|\cC(n,c)|=|\cC(n,1-c)|$.

  Now let $c\in(0,\frac12]$ and $\gamma>0$ be given. Set~$k'=5$.
  Lemma~\ref{lem:Counting} with input $c$ and $\gamma/2$ provides constants
  $\eps_{\subref{lem:Counting}}$, $d_0$, $n_{\subref{lem:Counting}}$, $k_0$.
  Set $d=d_0$. From Lemma~\ref{lem:GreyTriangle} with input~$d$ we obtain
  constants $\eps_{\subref{lem:GreyTriangle}}$ and $\eps'_0$.  Set
  $\eps:=\min\{\eps_{\subref{lem:Counting}},\eps_{\subref{lem:GreyTriangle}}\}$
  and $\eps':=\eps'_0$. The type lemma, Lemma~\ref{lem:type}, finally, with input
  $\eps$, $\eps'$, and $k_0$, $k'$ gives constants $k_1$ and
  $n_{\subref{lem:type}}$. We set $n_0=\max\{n_{\subref{lem:type}},n_{\subref{lem:Counting}},
  \frac{3}{\sqrt{\gamma}}k_1\}$.

  Now, for each graph~$G\in\cC(n,c)$ we apply the type lemma,
  Lemma~\ref{lem:type}, with parameters $\eps$, $\eps'$, $k_0$, $k'$ and $d$ and
  obtain an $(\eps,\eps',d,k')$-type $R$ of~$G$ on $k\le k_1$ vertices and with
  at most $\eps k^2$ non-edges.  Let~$\tilde\cR$ be the set of types obtained from these
  applications of Lemma~\ref{lem:type}.  It follows that $|\tilde\cR|\le
  4^{\binom{k_1}{2}}2^{k_1}\le 2^{k_1^2}$. 
  By Lemma~\ref{lem:GreyTriangle} applied with $d$, $\eps$, and $\eps'$, no
  coloured graph in~$\tilde\cR$ contains a triangle with three grey edges. Hence by
  Lemma~\ref{lem:Counting} applied with $c$, $\gamma/2$, $\eps$, $\eps'$ and $d$ we have
  $|\cR(R,\eps,\eps',d,n,c)| \le 2^{h(c)\binom n2 + \frac12\gamma \binom n2}$. Since,
  by Lemma~\ref{lem:type},
  \[ 
        \cC(n,c) \subset \bigcup_{R\in\tilde\cR}
        \cR(R,\eps,\eps',d,c,n)\;. 
   \] 
   we conclude from the choice of~$n_0$ that
   \begin{equation*}
     |\cC(n,c)|
     \le 2^{k_1^2} \cdot 2^{h(c)\binom n2 + \frac12 \gamma
     \binom n2} \le 2^{h(c)\binom n2 + \gamma \binom n2}
     \,.
   \end{equation*}
\end{proof}

\section{The proof of Theorem~\ref{thm:C5LinHom}}
\label{sec:linear}

Our proof of Theorem~\ref{thm:C5LinHom} consists of the following steps.  We
start, similarly as in the proof of Theorem~\ref{thm:main}, by constructing for
each graph~$G$ in $\Forb^*_{n,\eta}(C_5,c)$ a type~$R$ of size independent
of~$n$ with the help of the type lemma, Lemma~\ref{lem:type}. Next, we consider
each cluster~$V_i$ of a partition of $V(G)$ corresponding to~$R$ separately. We
shall show that the fact that~$G$ does not contain homogeneous sets of size
$\eta n$ implies that~$G[V_i]$ has many vertex disjoint induced copies of $P_3$,
the path on three vertices, or many vertex disjoint induced copies of the
anti-path $\antiP$, the complement of~$P_3$ (see
Lemma~\ref{lem:P3orNotP3}). Many induced copies of~$P_3$ or~$\antiP$ in two
clusters~$V_i$ and~$V_j$, however, limit the number of possibilities to insert
edges between~$V_i$ and~$V_j$ without inducing a~$C_5$ (see
Lemma~\ref{lem:C5Probability}). Combining this with the proof strategy from
Theorem~\ref{thm:main} will give us an upper bound for the number of graphs from
$\Forb^*_{n,\eta}(C_5,c)$ with type~$R$ (see Lemma~\ref{lem:NoLinHomSet}).
Finally, comparing this upper bound with the lower bound on
$|\Forb^*_{n}(C_5,c)|$ from Theorem~\ref{thm:main} will lead to the desired
result.

\smallskip

We start by proving that graphs without big homogeneous sets contain many vertex
disjoint induced~$P_3$ or~$\antiP$.

\begin{lemma}\label{lem:P3orNotP3}
  Let $G$ be a graph of order $n$ with $\hom(G) \le n/6$. Then one of the
  following is true.
  \begin{enumerate}[label=\rom]
    \item $G$ contains $n/6$ vertex disjoint induced copies of $P_3$, or
    \item $G$ contains $n/6$ vertex disjoint induced copies of $\antiP$.
  \end{enumerate}
\end{lemma}
\begin{proof}
  Let~$G$ be an $n$-vertex graph with $\hom(G) \le n/6$. Select a maximal set of
  disjoint copies of $P_3$. If this set consists of less than $n/6$ paths then
  there is a subgraph $G'\subseteq G$ with $v(G')=n/2$ that has no induced
  $P_3$ and thus is a vertex disjoint union of cliques $Q_1,\dots,Q_\ell$.
  We claim that in~$G'$ we can find $n/6$ vertex disjoint
  induced~$\antiP$, which proves the lemma.

  Indeed, since $\hom(G)\le n/6$ we have $\ell\le n/6$ and for each $i\in[\ell]$
  we have $q_i:=|Q_i|\le n/6$.  This implies $\sum_{i\in[\ell]} \lfloor q_i/2
  \rfloor\ge \frac12(\frac n2-\frac n6)=n/6$. It follows that we can find a set
  of $n/6$ vertex disjoint edges $E=\{e_1,\dots,e_{n/6}\}$ in these cliques in
  the following way.
  We first choose as
  many vertex disjoint edges in~$Q_1$ as possible, then in~$Q_2$, and so on,
  until we chose $n/6$ edges in total. Let~$Q_k$ be the last clique used in this
  process. Then for each clique~$Q_i$ with $i<k$ at most one vertex was unused
  in this process, and in~$Q_k$ possibly several vertices were unused. Let~$X$
  be the set of all these unused vertices together with all
  vertices from $\bigcup_{k<i\le\ell}Q_i$. Clearly $|X|=n/6$.

  We consider the auxiliary bipartite graph $B=(X\cup E, E_B)$ with
  $\{x,e\}\in E_B$ for $x\in X$ and $e\in E$ iff~$x$ and~$e$ do not lie in the
  same clique of~$G'$. Clearly, for each vertex~$x\in Q_i$ with $i\in[\ell]$
  there are at most $(n/6)/2$ edges from~$E$ which lie in~$Q_i$. Hence
  $\deg_B(x)\ge \frac12|E|$ for all~$x\in X$, and $|X|=|E|$. It follows that~$B$
  has a perfect matching, which means that there are $n/6$ vertex disjoint
  induced~$\antiP$ in~$G'$ as claimed.
\end{proof}

Now suppose we are given a graph~$G$ with vertex set $V_1\dcup V_2$ and no edges
between $V_1$ and $V_2$. Let further $H_1$ and $H_2$ be such that for $i\in[2]$
the graph~$H_i$ induces a copy of~$P_3$ or~$\antiP$ in $G[V_i]$. Observe that,
no matter which combination of $P_3$ or $\antiP$ we choose, we can create an
induced~$C_5$ in~$G$ by adding appropriate edges between~$H_1$ and~$H_2$.  Since
we are interested in graphs without induced~$C_5$ this motivates why we call
$(H_1,H_2)$ a \emph{dangerous pair} of $(V_1,V_2)$.

Our next goal is to use these dangerous pairs in order to derive an upper bound
on the number of possibilities to insert edges between~$V_1$ and~$V_2$ without
creating an induced copy of~$C_5$ if we know that $(V_1,V_2)$ contains many
dangerous pairs. In order to quantify this upper bound in
Lemma~\ref{lem:C5Probability} we use the following technical definition.
We define $R(c)=c^4(1-c)^4$ and the function $r:(0,1)\to\R^+$ with
\begin{equation}\label{eq:defr}
 r(c) = \frac{1}{72}\begin{cases} R(2c) & \text{if $c<\tfrac 14$,}\\        
        (1/4)^{4} & \text{if $c\in [\tfrac 14,\tfrac 34]$,}\\
 		R(2c-1) & \text{otherwise.} \end{cases} 
\end{equation}
Recall in addition the definition of the function~$h(c)$ from~\eqref{eq:h}.


\begin{lemma}\label{lem:C5Probability} 
  For every $0<c_0\le\frac12$ there is~$n_0$ such that for all~$c$ with $c_0\le
  2c\le 1-c_0$ and $n\ge n_0$ the following holds. Let $G_1=(V_1,E_1)$ and
  $G_2=(V_2,E_2)$ be two $n$-vertex graphs, each of which contains $n/6$ vertex
  disjoint induced copies of $P_3$ or $n/6$ vertex disjoint induced copies of
  $\antiP$. Let~$G=(V_1\dcup V_2,E)$ be the disjoint union of~$G_1$ and~$G_2$.
  Then there are at most 
  \[2^{2n^2\bleft( h(c)-r(c) \right)}\] 
  ways to add exactly $2c n^2$ edges to~$G$ that run between $V_1$ and $V_2$
  without inducing a $C_5$ in $G$.
\end{lemma}

We remark that in the proof of this lemma we are going to make use of the
following probabilistic principle: We can \emph{count} the number of elements in
a finite set~$X$ which have some property~$P$, by determining the
\emph{probability} that an element which is chosen from~$X$ uniformly at random
has property~$P$.

\begin{proof}[Proof of Lemma~\ref{lem:C5Probability}]
  Given~$c_0\in(0,\frac12]$ let $n_0$ be sufficiently large such that
  \begin{equation}\label{eq:c_0}
    n_0^2 e^{- 2r(c_0/2) n_0^2} \le 2^{-2r(c_0/2) n_0^2}\,.
  \end{equation}
  Now let~$c$ be such that $c_0\le 2c\le 1-c_0$. Observe first that it suffices
  to prove the lemma for $2c\le\frac12$, since induced $C_5$-free graphs are
  self-complementary and~$P_3$ is the complement of~$\antiP$. Hence, we assume
  from now on that $2c\le\frac12$.  Observe moreover, that~\eqref{eq:c_0}
  remains valid if~$c_0$ is replaced by~$c$ since~$r(c)$ is monotone increasing
  in~$[0,\frac12]$.  Let~$G_1$, $G_2$, and $G$ be as required.

  Our first goal is to estimate the probability $P^*$ of inducing no~$C_5$
  in~$G$ when choosing uniformly at random exactly $2c n^2$ edges between $V_1$
  and $V_2$.  Instead of dealing with~$P^*$ directly, we consider the following
  binomial random graph $\cG(V_1,V_2,p)$ with $p=2c\,$: we start with~$G$ and add
  each edge between $V_1$ and $V_2$ independently with probability~$p$.

  Now, let~$A$ be the event that $\cG(V_1,V_2,p)$ contains exactly $2cn^2$ edges
  between~$V_1$ and~$V_2$, and let~$B$ be the event that~$\cG(V_1,V_2,p)$
  contains no induced~$C_5$. Observe that each graph with~$2cn^2$ edges
  between~$V_1$ and~$V_2$ is equally likely in~$\cG(V_1,V_2,p)$ and thus
  \begin{equation} \label{eq:P*}
    P^* \le \Prob[B|A] \le \frac{\Prob[B]}{\Prob[A]} \,.
  \end{equation}
  Hence it suffices to estimate~$\Prob[A]$ and~$\Prob[B]$.

  We first bound~$\Prob[B]$.
  By assumption there are at least $n^2/36$ dangerous pairs in
  $(V_1,V_2)$. Now fix such a dangerous pair $(H_1,H_2)$.
  The probability that $(H_1,H_2)$ induces a~$C_5$ in $\cG(V_1,V_2,p)$
  is at least $p^2(1-p)^4$ unless~$H_1$ and~$H_2$ are both~$\antiP$, and at least
  $p^4(1-p)^2$ unless~$H_1$ and~$H_2$ are both~$P_3$.
  Thus $(H_1,H_2)$ induces a copy of $C_5$ with probability at least
  \begin{equation*}
    p^4(1-p)^4 
    = (2c)^4(1-2c)^4
    \geByRef{eq:defr} 72\cdot r(c) \,.   
  \end{equation*}
  Since we can lower bound the probability of~$B$ by the probability that none of the $n^2/36$ dangerous
  pairs in $(V_1,V_2)$ induces a~$C_5$ in $\cG(V_1,V_2,p)$ we obtain
  \begin{equation*}
    \Prob[B] \le \big(1- 72\cdot r(c)\big)^{n^2/36}  \le e^{-2r(c) n^2}.
  \end{equation*}
  Moreover, since the expected number of edges between~$V_1$ and~$V_2$
  in~$\cG(V_1,V_2,p)$ is $2cn^2$ we clearly have $\Prob[A] \ge
  n^{-2}$. By the choice of~$n_0$, combining this with~\eqref{eq:P*} gives
  \begin{equation} \label{eq:C5Prob1} 
    P^* \le \frac{e^{- 2r(c) n^2}}{n^{-2}} \leByRef{eq:c_0} 2^{-2r(c) n^2}.
  \end{equation}

  It remains to estimate the number~$N$ of ways to
  choose exactly $2c n^2$ edges between~$V_1$ and~$V_2$. We have
  \begin{equation}\label{eq:C5Prob2} 
    N\
    \le \binom{n^2}{2c n^2} 
    \leByRef{eq:BinCoe} 2^{2h(c) n^2 } \,.
  \end{equation}
  This implies that the number of ways to add exactly $2c n^2$ edges to~$G$ that
  run between $V_1$ and $V_2$ without inducing a~$C_5$ is 
  \begin{equation*}
    P^*\cdot N 
    \leBy{\eqref{eq:C5Prob1},\eqref{eq:C5Prob2}}
    2^{-2r(c) n^2} \cdot 2^{2h(c) n^2 } \,.
  \end{equation*}
\end{proof}

Next, we want to show that Lemma~\ref{lem:C5Probability} allows us to derive an
upper bound on the number of graphs~$G$ such that (a) $G$ has no large
homogeneous sets and (b) $G$ has a fixed type~$R$ which does not
contain a triangle with three grey edges. Our aim is to obtain an upper bound
which is much smaller than the bound provided in Lemma~\ref{lem:Counting} for
the corresponding problem without restriction (a). Lemma~\ref{lem:NoLinHomSet}
states that this is possible. Recall for this purpose the definition of
$\cR(R,\eps,\eps',d,k',n,c)$ from~\eqref{eq:cR}.

\begin{lemma}\label{lem:NoLinHomSet}
  For every $c$ with $0<c\le\frac12$, and every
  $\gamma>0$ there exist $\eps_0,d_0>0$ and an integer $k_0$ such
  that for all integers $k_1\ge k_0$, $k'$ there is an integer $n_0$ such that
  for all positive $d\le d_0$, $\eps\le\eps_0$, $\eps'$, and all integers
  $n\ge n_0$ and $k_0\le k\le k_1$ the following holds. If $R$ is a coloured graph of
  order $k$ which has at most $\eps k^2$ non-edges and does not contain a
  triangle with three grey edges, and $\eta=1/(6k_1)$, then
  \[
    |\cR(R,\eps,\eps',d,k',n,c) \cap \Forb^*_{n,\eta}(C_5,c) | \le 
    2^{\bleft( h(c)-r(c) \right)\binom n2 + \gamma \binom n2} \,.
  \]
\end{lemma}

In the proof of this lemma we combine the strategy of the proof of
Lemma~\ref{lem:Counting} with an application of Lemma~\ref{lem:P3orNotP3} to all
clusters of a partition corresponding to~$R$, and an application of
Lemma~\ref{lem:C5Probability} to regular pairs of medium density.
We shall make use of the following observation.

Using the definition of~$r(c)$ from~\eqref{eq:defr}, it is easy to check that
$f(c):=h(c)-r(c)$ is a concave function for $c \in (0,1)$. Thus~$f$ enjoys the
following property, which is a special form of Jensen's inequality (see, e.g.,
\cite{HarLitPol52}).

\begin{proposition}[Jensen's inequality]\label{lem:Convexity}
  Let $f$ be a concave function, $0<c<1$, $0<c_i<1$ for $i\in [m]$ and let
  $\sum_{i=1}^{m} c_i = m c$. Then 
  \[
	\sum\limits_{i=1}^{m} f(c_i) 
	\le m \cdot f(c) \,. 
  \eqed
  \]
\end{proposition}

\begin{proof}[Proof of Lemma~\ref{lem:NoLinHomSet}]
  Let $c,\gamma>0$ be given and choose $\eps_0$, $d_0$,
  $k_0$ and $n_0$ as in the proof of Lemma~\ref{lem:Counting}. Let~$d\le d_0$ and~$k_1$ be given and
  possibly increase~$n_0$ such that $n_0\ge 2k_1
  n_{\subref{lem:C5Probability}}$, where $n_{\subref{lem:C5Probability}}$ is
  the constant from Lemma~\ref{lem:C5Probability} with parameter~$d$, and
  such that
  \begin{equation}\label{eq:C5P:n0}
    \frac34 k_1^2\log n_0 + 2n_0 \le \frac{\gamma}{10} \binom{n_0}{2}
    \,.
  \end{equation} 
  If necessary decrease $\eps_0$ such that 
  \begin{equation}\label{eq:C5P:eps}
  	3\eps_0 \log \frac 1c \le \frac{\gamma}{10}\, .
  \end{equation}
  Let~$\eps\le\eps_0$, $\eps'$, $n\ge n_0$, and~$k$ with $k_0\le k\le k_1$ be
  given.

  Let~$R=([k],E_R,\sigma)$ be a coloured graph which has at most $\eps k^2$
  non-edges and does not contain a triangle with three grey edges.  In the
  proof of Lemma~\ref{lem:Counting} we counted the number of graphs in
  $\cR(R,\eps,\eps',d,k',n,c)$ by estimating the number of equipartitions
  $V_0\dcup\dots\dcup V_k$ of $[n]$, the number of choices for edges with
  one end in the exceptional set~$V_0$ and edges in pairs $(V_i,V_j)$ such
  that $\{i,j\}\not\in E_R$, the number of choices for edges inside
  clusters~$V_i$ (such that $i$ is white or black in~$R$), and the number of
  choices for at most $c\binom n2$ edges in pairs $(V_i,V_j)$ such that
  $\{i,j\}$ is a white, black, or grey edge of~$R$. Now we are interested
  in the number of graphs in $\cR(R,\eps,\eps',d,k',n,c) \cap
  \Forb^*_{n,\eta}(C_5,c)$. Clearly, we can use the same strategy, and it
  is easy to verify that the estimates in
  \eqref{eq:Count0}--\eqref{eq:Count2} and thus in~\eqref{eq:count:prod}
  from the proof of Lemma~\ref{lem:Counting} remain valid in this setting.
  Form now on, as in the proof of Lemma~\ref{lem:Counting}, we fix a
  partition $V_0\dcup\dots\dcup V_k$ of~$[n]$ and observe that also
  $m_g:=|E_g|\le\tfrac{k^2}{4}$ still holds for the grey edges~$E_g$
  in~$R$.  However, we shall now use Lemma~\ref{lem:C5Probability} to
  obtain an improved bound on the number of possible choices for edges in
  $E_g$-pairs $(V_i,V_j)$, and use this to replace~\eqref{eq:Count6} by a
  smaller bound on the number~$N_g$ of possible ways to distribute at most
  $c\binom n2$ edges to $E_g$-pairs.  Since in the following we do not rely
  on any interferences between different $E_g$-pairs, clearly~$N_g$ will be
  maximal if~$m_g$ is maximal, and hence we assume from now on that
  \begin{equation}\label{eq:C5P:mg}
    m_g=\tfrac{k^2}{4}
    \,.
  \end{equation}
  Let $s:=|V_1|=\cdots=|V_k|$ and observe that 
  \begin{equation}\label{eq:C5P:s}
    \frac{n}{k} \ge s\ge(1-\eps)\frac{n}{k} \ge n_{\subref{lem:C5Probability}}
    \,.
  \end{equation}
  By Lemma~\ref{lem:P3orNotP3}, for each cluster~$V_i$ of a partition~$P$ of a graph
  in $\cR(R,\eps,\eps',d,k',n,c) \cap \Forb^*_{n,\eta}(C_5,c)$ such that~$P$
  corresponds to~$R$, we have that~$V_i$ contains either $s/6$ copies of~$P_3$
  or $s/6$ copies of~$\antiP$. Hence we will assume from now on, that in our
  fixed partition the clusters~$V_i$ have this property.

  We now upper bound~$N_g$ by multiplying the possible ways~$A$ to assign at
  most $c\binom n2$ edges to one $E_g$-pair each, and the maximum number~$B$ of
  ways to choose all these assigned edges in the corresponding pairs,
  without inducing a~$C_5$. First observe that we have
 \begin{equation}\label{eq:CP5:A}
   A\le \bigg(c\binom n2\bigg)^{m_g+1} \le n^{3m_g} \leByRef{eq:C5P:mg} 2^{\frac 34
     k^2\log n}
   \,.
 \end{equation}
 For estimating~$B$, we now assume that we fixed an assignment in which
 each pair $(V_i,V_j)$ with $\{i,j\}\in E_g$ is assigned $2c_{ij}s^2$
 edges. Let $\hat c$ be such that
 \begin{equation}\label{eq:CountFine0}
   \sum_{\{i,j\}\in E_g} 2c_{ij} s^2 =: \hat c n^2 \le c \binom n2. 
 \end{equation}
 Observe further that, since $\{i,j\}\in E_g$ is a grey edge of~$R$, and we are
 interested into counting graphs with a partition corresponding to~$R$ we
 can assume that $d\le 2c_i\le (1-d)$.
 Hence, by~\eqref{eq:C5P:s} we can apply Lemma~\ref{lem:C5Probability}
 with $c_0=d$ to infer that for each $\{i,j\}\in E_g$ there are at most
 \begin{equation}\label{eq:C5P:Bij}
   B_{ij} \le 2^{2s^2 \bleft(h(c_{ij}) - r(c_{ij})\right) } 
 \end{equation}
 possible ways to choose the $2c_{ij}s^2$ edges in $(V_i,V_j)$
 without inducing a~$C_5$.
 Now let $2\tilde c:=\sum_{\{i,j\} \in E_g} 2c_{i,j} / m_g$ and observe that
 \begin{equation*}
   \tilde c
   \eqByRef{eq:CountFine0} \hat c \frac{n^2}{2s^2 m_g}
   \leByRef{eq:C5P:s} \hat c \frac{k^2}{2(1-\eps)^2 m_g}
   \leByRef{eq:C5P:mg} 2(1+3\eps)\hat c 
   \leByRef{eq:CountFine0} (1+3\eps)c \le \frac{3}{4}
 \end{equation*}
 Therefore, since $f(x):=h(x)-r(x)$ is a concave function for $x \in (0,1)$, which is
 moreover non-decreasing for $x\le\frac34$ we can infer from
 Lemma~\ref{lem:Convexity} that
 \begin{equation} \label{eq:C5P:sum}
   \sum_{\{i,j\} \in E_g} f(c_{ij})
   \le m_g \cdot f(\tilde c)
   \le m_g \cdot f\big(c(1+3\eps)\big)
   \eqByRef{eq:C5P:mg} \frac{k^2}{4} f\big( c(1+3\eps) \big)\,.
 \end{equation}
 As $h(x)$ is a convex function with $h'(x)\le \log(1/x)$ and $r(x)$ is
 non-decreasing for $x\le 3/4$ we have
 \begin{equation*}
 	f\big(c(1+3\eps)\big) \le h(c+3\eps)-r(c) \le h(c) + 3\eps h'(c) - r(c)
 	\leByRef{eq:C5P:eps} h(c)-r(c) + \frac{\gamma}{10}\, . 
 \end{equation*}
 Together with~\eqref{eq:C5P:Bij} and~\eqref{eq:C5P:sum} this implies
 \begin{equation*}
   B=\prod_{\{i,j\}\in E_g} B_{ij}
   \le 2^{2 s^2 (k^2/4) \bleft( h(c)-r(c)+\gamma/10\right)}
   \leByRef{eq:C5P:s} 2^{(n^2/2) \bleft( h(c)-r(c) +\gamma/10\right)}
   \,,
 \end{equation*}
 which in turn together with~\eqref{eq:CP5:A} gives
 \begin{equation*} 
   N_g 
   \le 2^{\frac 34 k^2\log n}
     \cdot  2^{(n^2/2) \bleft( h(c)-r(c) +\gamma/10\right)}
   \leByRef{eq:C5P:n0} 2^{\bleft( h(c)-r(c) \right)\binom n2+(\gamma/5) \binom n2} \,.
  \end{equation*}
  By multiplying this with~\eqref{eq:count:prod} from the proof of
  Lemma~\ref{lem:Counting} we obtain
  \[
     |\cR(R,\eps,\eps',d,k',n,c) \cap \Forb^*_{n,1/6k}(C_5,c) | \le 
     2^{\bleft( h(c)-r(c) \right)\binom n2 + \gamma \binom n2} \,,
   \]
 as claimed.
\end{proof}

Lemma~\ref{lem:NoLinHomSet} together with the type lemma, Lemma~\ref{lem:type},
implies an upper bound on $|\Forb_{n,\eta}^*(C_5,c)|$.  Now we can combine
this with the lower bound on $|\Forb^*_n(C_5,c)|$ which follows from
Lemma~\ref{lem:lower} in order to prove Theorem~\ref{thm:C5LinHom}.

\begin{proof}[Proof of Theorem~\ref{thm:C5LinHom}]
  Observe first that, since~$C_5$ is self-complementary, it suffices to
  prove Theorem~\ref{thm:C5LinHom} for $c\le 1/2$. Hence we assume $c\le
  1/2$ from now on.

  We first need to set up some constants.  Given~$c\in(0,\frac12]$, we choose
  $\gamma>0$ such that $2\gamma<r(c)$.  For input~$c$ and~$\gamma$
  Lemma~\ref{lem:lower} supplies us with a constant~$n_{\subref{lem:lower}}$.  We
  apply Lemma~\ref{lem:NoLinHomSet} with input~$c$ and~$\gamma/2$ to obtain
  $\eps_{\subref{lem:NoLinHomSet}}$, $d_0$, and $k_0$. Next, we apply
  Lemma~\ref{lem:GreyTriangle} with input~$d_0$ and obtain constants
  $\eps_{\subref{lem:GreyTriangle}}$ and $\eps'$.  Let
  $\eps:=\min\{\eps_{\subref{lem:NoLinHomSet}},\eps_{\subref{lem:GreyTriangle}}\}$.
  For input~$\eps$, $\eps'$, and~$k_0$ Lemma~\ref{lem:type} returns
  constants~$k_1$ and~$n_{\subref{lem:type}}$.  With this parameter~$k_1$ we
  continue the application of Lemma~\ref{lem:NoLinHomSet} and
  obtain~$n_{\subref{lem:NoLinHomSet}}$. Choose
  $n_0:=\max\{n_{\subref{lem:lower}},n_{\subref{lem:type}},n_{\subref{lem:NoLinHomSet}},\tfrac{3}{\sqrt{\gamma}}k_1\}$,
  assume that~$n\ge n_0$, and set $\eta:=1/(6k_1)$.

  Now, for each graph~$G\in \Forb_{n,\eta}^*(C_5,c)$ we apply the type
  lemma, Lemma~\ref{lem:type}, with parameters $\eps$, $\eps'$, $k_0$, $k'$ and
  $d$ and obtain an $(\eps,\eps',d,k')$-type $R$ of~$G$ on $k$ vertices with
  $k_0\le k\le k_1$ and with at most $\eps k^2$ non-edges.  Let~$\tilde\cR$ be
  the set of types obtained from these applications of
  Lemma~\ref{lem:type}. It follows that $|\tilde\cR|\le
  4^{\binom{k_1}{2}}2^{k_1}\le 2^{k_1^2}$. By Lemma~\ref{lem:GreyTriangle}
  applied with $d_0$, $\eps$, and $\eps'$, no coloured graph in~$\tilde\cR$ contains a triangle with three grey edges. Hence
  by Lemma~\ref{lem:NoLinHomSet} applied with $c$, $\gamma/2$, $\eps$,
  $\eps'$, $n$, and $k$ we have
  \begin{equation*}
    \Big|\,\cR(R,\eps,\eps',d,k',n,c) \,\cap\, \Forb^*_{n,\eta}(C_5,c) \,\Big|
    \le 2^{\bleft( h(c)-r(c) \right)\binom n2 + \frac12 \gamma \binom n2} \,.
  \end{equation*} 
  Since, by Lemma~\ref{lem:type},
  \[ 
        \Forb^*_{n,\eta}(C_5,c) \subset \bigcup_{R\in\tilde\cR}
        \Big( \cR(R,\eps,\eps',d,k',c,n) \,\cap\, \Forb^*_{n,\eta}(C_5,c) \Big)
   \] 
   we conclude from the choice of~$n_0$ that
   \begin{equation*}
     |\Forb^*_{n,\eta}(C_5,c)|
     \le 2^{k_1^2} \cdot 2^{\bleft( h(c)-r(c) \right)\binom n2
     + \frac12 \gamma \binom n2} \le 2^{\bleft( h(c)-r(c) \right)\binom n2 + \gamma \binom n2}
    \,.
   \end{equation*}
   On the other hand, by Lemma~\ref{lem:lower} and the choice of~$n_0$ we have
   \begin{equation*}
     |\Forb^*_n(C_5,c)| \ge 2^{h(c) \binom n2 - \gamma\binom n2}
     \,.
   \end{equation*}
   Since $2\gamma<r(c)$ by the choice of~$\gamma$,
   this implies that almost all graphs
   in $\Forb^*_n(C_5,c)$ satisfy $\hom(G)\ge \eta n$.
\end{proof}

%
% ...





\bibliographystyle{amsplain_yk} 
\bibliography{perfectgraphs}

\end{document}